\documentclass[a4paper, 12pt, reqno]{amsart}

\usepackage{cite}
\usepackage{amsmath, amsfonts, amssymb, amsthm, txfonts, mathtools}
\usepackage[all]{xy}
\usepackage{bigints}
\usepackage{textcomp}

\newtheorem{theorem}{Theorem}[section]
\newtheorem{acknowledgement}[theorem]{Acknowledgement}

\newtheorem{corollary}[theorem] {Corollary}
\newtheorem{definition}[theorem]{Definition}

\newtheorem{lemma} [theorem]{Lemma}

\newtheorem{proposition}[theorem]{Proposition}
\newtheorem{remark}[theorem]{Remark}
\newtheorem{workinghypothesis}[theorem]{Working Hypothesis}
\newcommand{\Z}{\mathbb{Z}}

\newcommand{\C}{\mathbb{C}}

\newcommand{\GL}{{\rm GL}}
\newcommand{\SL}{{\rm SL}}
\newcommand{\Hom}{{\rm Hom}}
\newcommand{\ind}{{\rm ind}}
\newcommand{\Ind}{{\rm Ind}}
\newcommand{\Ext}{{\rm Ext}}
\newcommand{\vol}{{\rm vol}}

\title{Branching laws for the metaplectic cover of $\GL_{2}$}
\author{Shiv Prakash Patel}
\address{Department of Mathematics, Indian Institute of Technology Bombay, Mumbai - 400076}
\email{shiv@math.iitb.ac.in}
\date{\today}


\begin{document}
\maketitle

\begin{abstract}
 Let $F$ be a non-Archimedian local field of characteristic zero and $E/F$ a quadratic extension. 
 The aim of the present article is to study the multiplicity of an irreducible admissible representation of $\GL_2(F)$ occurring in an irreducible admissible genuine representation of non-trivial two fold covering $\widetilde{\GL}_2(E)$ of $\GL_2(E)$. 
\end{abstract}

\section{Introduction}  \label{introduction: DP-metaplectic}
Let $F$ be a non-Archimedian local field of characteristic zero and let $E$ be a quadratic extension of $F$. 
The problem of decomposing a representation of $\GL_2(E)$ restricted to $\GL_2(F)$ was considered and solved by D. Prasad in \cite{Prasad92}, proving a multiplicity one theorem, and giving an explicit classification of representations $\pi_1$ of $\GL_2(E)$ and $\pi_2$ of $\GL_2(F)$ such that there exists a non-zero $\GL_2(F)$ invariant linear form:
\[
 l : \pi_1 \otimes \pi_2 \rightarrow \C.
\]
This problem is closely related to a similar branching law from $\GL_{2}(E)$ to $D_{F}^{\times}$, 
where $D_{F}$ is the unique quaternion division algebra which is central over $F$, and $D_{F}^{\times} \hookrightarrow \GL_{2}(E)$. We recall that the embedding $D_{F}^{\times} \hookrightarrow \GL_{2}(E)$ is given by fixing an isomorphism $D_{F} \otimes E \cong M_{2}(E)$, by the Skolem-Noether theorem, 
which is unique up to conjugation by elements of $\GL_{2}(E)$. 
Henceforth, we fix one such embedding of $D_{F}^{\times}$ inside $\GL_{2}(E)$. 
The restriction problems for the pair $(\GL_{2}(E), \GL_{2}(F))$ and $(\GL_{2}(E), D_{F}^{\times})$ are related by a certain dichotomy. 
More precisely, the following result was proved in \cite{Prasad92} :
\begin{theorem}[D. Prasad]
Let $\pi_{1}$ and $\pi_2$ be irreducible admissible infinite dimensional representations of $\GL_{2}(E)$ and $\GL_{2}(F)$ respectively such that the central character of $\pi_1$ restricted to the center of $\GL_2(F)$ is the same as the central character of $\pi_2$.  
Then
\begin{enumerate}
 \item For a principal series representation $\pi_2$ of $\GL_2(F)$, we have
 \[
\dim \Hom_{\GL_2(F)} \left( \pi_{1},\pi_{2} \right) = 1.
\]
 \item For a discrete series representation $\pi_2$ of $\GL_2(F)$, 
 let $\pi_2'$ be the finite dimensional representation of $D_{F}^{\times}$ associated to $\pi_2$ by the Jacquet-Langlands correspondence, then
 \[
  \dim \Hom_{\GL_2(F)} \left( \pi_{1},\pi_{2} \right) + \dim \Hom_{D_{F}^{\times}} \left( \pi_{1}, \pi_{2}' \right) = 1.
 \]
\end{enumerate}
\end{theorem}
In this paper, we will study the analogous problem in the metaplectic setting. 
More precisely, instead of considering $\GL_2(E)$ we will consider the group $\widetilde{\GL}_{2}(E)_{\C^{\times}}$ which is a topological central extension of $\GL_{2}(E)$ by $\C^{\times}$, which is obtained from the two fold topological central extension $\widetilde{\GL}_{2}(E)$ described below. 
We recall that there is unique (up to isomorphism) two fold cover of $\SL_{2}(E)$ called the metaplectic cover and denoted by $\widetilde{\SL}_{2}(E)$ in this paper, but there are many inequivalent two fold coverings of $\GL_{2}(E)$ which extend this two fold covering of $\SL_{2}(E)$. We fix a covering of $\GL_{2}(E)$ as follows. 
Observe that $\GL_{2}(E)$ is a semi-direct product of $\SL_{2}(E)$ and $E^{\times}$, where $E^{\times}$ sits inside $\GL_{2}(E)$ by $e \mapsto \left( \begin{matrix} e & 0 \\ 0 & 1 \end{matrix} \right)$. 
The action of $E^{\times}$ on $\SL_{2}(E)$ lifts to an action on $\widetilde{\SL}_{2}(E)$.
Denote $\widetilde{\GL}_{2}(E) = \widetilde{\SL}_{2}(E) \rtimes E^{\times}$ which we call `the' metaplectic cover of $\GL_{2}(E)$. 
This cover can be described by an explicit 2-cocycle on $\GL_{2}(E)$ with values in $\{ \pm 1 \}$, see \cite{Kubota69}.
The group $\widetilde{\GL}_{2}(E)$ is a topological central extension of $\GL_{2}(E)$ by $\mu_{2} := \{ \pm 1 \}$, i.e., we have an exact sequence of topological groups:
\[
 1 \rightarrow \mu_{2} \rightarrow \widetilde{\GL}_{2}(E) \rightarrow \GL_{2}(E) \rightarrow 1.
\]
The group $\widetilde{\GL}_{2}(E)_{\C^{\times}} := \widetilde{\GL}_{2}(E) \times_{\mu_2} \C^{\times}$ is called the $\C^{\times}$-cover of $\GL_{2}(E)$ obtained from the two fold cover $\widetilde{\GL}_{2}(E)$, and is a topological central extension of $\GL_{2}(E)$ by $\C^{\times}$, i.e., we have an exact sequence of topological groups:
\[
 1 \rightarrow \C^{\times} \rightarrow \widetilde{\GL}_{2}(E)_{\C^{\times}} \rightarrow \GL_{2}(E) \rightarrow 1.
\]
Now we recall the following result regarding splitting of this cover when restricted to certain subgroups. 
This makes it possible to consider an analog of the Prasad's restriction problem in the metaplectic case.
\begin{theorem} {\rm \cite{Patel14}} \label{splitting theorem}
Let $E$ be a quadratic extension of a non-Archimedian local field and $\widetilde{\GL}_{2}(E)$ be the two-fold metaplectic covering of $\GL_{2}(E)$.  
Then:
\begin{enumerate}
\item The two fold metaplectic covering $\widetilde{\GL}_{2}(E)$ splits over the subgroup $\GL_{2}(F)$.
\item The $\C^{\times}$-covering obtained from $\widetilde{\GL}_{2}(E)$ splits over the subgroup $D_{F}^{\times}$. 
\end{enumerate} 
\end{theorem}

Note that the splittings over $\GL_{2}(F)$ and $D_{F}^{\times}$ in Theorem \ref{splitting theorem} are not unique. 
As there is more than one splitting in each case,
to study the problem of decomposing a representation of $\widetilde{\GL}_{2}(E)_{\C^{\times}}$ restricted to $\GL_{2}(F)$ and $D_{F}^{\times}$, 
we must fix one splitting of each of the subgroups ${\rm GL}_2(F)$ and $D_{F}^{\times}$, 
which are related to each other. 
We make the following working hypothesis, which has been formulated by D. Prasad. 
\begin{workinghypothesis} \label{working hypothesis}
Let $L$ be a quadratic extension of $F$. The sets of splittings
\begin{displaymath}
\xymatrix{ 
     & \widetilde{\GL}_{2}(E)_{\C^{\times}}  \ar[dd]   &  &   & \widetilde{\GL}_{2}(E)_{\C^{\times}} \ar[dd] \\ 
     &    & {\rm and}  & & \\
     \GL_{2}(F) \ar@{^{(}-->}[ruu]^{i} \ar[r] & \GL_{2}(E)   &  & D_{F}^{\times} \ar@{^{(}-->}[ruu]^{j} \ar[r] & \GL_{2}(E) }
\end{displaymath} 
are principal homogeneous spaces over the Pontrjagin dual of $F^{\times}$. 
There is a natural identification between these two sets of splittings in such a way that for any quadratic extension $L$ of $F$, any two embeddings of $L^{\times}$ in $\widetilde{\rm GL}_{2}(E)_{\C^{\times}}$ as in the following diagrams are conjugate in $\widetilde{\GL}_{2}(E)_{\C^{\times}}$. 
\begin{displaymath}
\xymatrix{ 
 &     & \widetilde{\GL}_{2}(E)_{\C^{\times}}  \ar[dd]         &  &   &     & \widetilde{\GL}_{2}(E)_{\C^{\times}} \ar[dd]   \\
 &  &  & {\rm and} & & & \\
L^{\times} \ar@{^{(}-->}[rruu] \ar[r]  &  \GL_{2}(F) \ar@{^{(}-->}[ruu] \ar[r] & \GL_{2}(E)  &  &  L^{\times} \ar@{^{(}-->}[rruu] \ar[r] & D_{F}^{\times} \ar@{^{(}-->}[ruu] \ar[r] & \GL_{2}(E) }
\end{displaymath} 
Here $L^{\times} \hookrightarrow \GL_{2}(F)$ (respectively, $L^{\times} \hookrightarrow D_{F}^{\times}$) are obtained by identifying a suitable maximal torus of $\GL_{2}(F)$ 
(respectively, $D_{F}^{\times}$ viewed as an algebraic group) with ${\rm Res}_{L/F} \mathbb{G}_{m}$. 
\end{workinghypothesis}

\begin{definition}
 A representation of $\widetilde{\GL}_{2}(E)$ (respectively $\widetilde{\GL}_{2}(E)_{\C^{\times}}$) is called genuine if $\mu_{2}$ acts non trivially (respectively $C^{\times}$ acts by identity).
\end{definition}

In particular, a genuine representation does not factor through $\GL_{2}(E)$. 
In  what follows, we always consider genuine representations of the metaplectic group $\widetilde{\GL}_{2}(E)$.
Let $B(E), A(E)$ and $N(E)$ be the Borel subgroup, maximal torus and maximal unipotent subgroup of ${\rm GL}_2(E)$ consisting of all upper triangular matrices, diagonal matrices and upper triangular unipotent matrices respectively. 
Let $B(F), A(F)$ and $N(F)$ denote the corresponding subgroups of ${\rm GL}_2(F)$. 
Let $Z$ be the center of ${\rm GL}_2(E)$ and $\tilde{Z}$ the inverse image of $Z$ in $\widetilde{\GL}_{2}(E)$. 
Note that $\tilde{Z}$ is an abelian subgroup of $\widetilde{\GL}_{2}(E)$ but is not the center of $\widetilde{\GL}_{2}(E)$; the center of $\widetilde{\GL}_{2}(E)$ is $\tilde{Z^2}$, the inverse image of $Z^2 := \{ z^{2} \mid z \in Z \}$. \\

Let $\psi$ be a non-trivial additive character of $E$. 
Note that the metaplectic covering splits when restricted to the subgroup $N(E)$ and hence $\psi$ gives a character of $N(E)$. 
Let $\pi$ be an irreducible admissible genuine representation of $\widetilde{\GL}_{2}(E)$ and $\pi_{N(E), \psi}$, the $\psi$-twisted Jacquet module which is a $\tilde{Z}$-module. 
Let $\omega_{\pi}$ be the central character of $\pi$.
A character of $\tilde{Z}$ appearing in $\pi_{N(E), \psi}$ agrees with $\omega_{\pi}$ when restricted to $\tilde{Z^2}$.
Let $\Omega(\omega_{\pi})$ be the set of genuine characters of $\tilde{Z}$ whose restriction to $\tilde{Z^2}$ agrees with $\omega_{\pi}$. 
We also realize $\Omega(\omega_{\pi})$ as a $\tilde{Z}$-module, i.e. as direct sum of characters in $\Omega(\omega_{\pi})$ with multiplicity one.  
From \cite[Theorem~ 4.1]{GHPS79}, one knows that the multiplicity of a character $\mu \in \Omega(\omega_{\pi})$ in the $\tilde{Z}$-module $\pi_{N(E), \psi}$ is at most one.
Hence $\pi_{N(E), \psi}$ is a $\tilde{Z}$-submodule of $\Omega(\omega_{\pi})$. 
Now we state the main result of this paper.

We abuse notation and write $\widetilde{\GL}_{2}(E)$ for $\widetilde{\GL}_{2}(E)_{\C^{\times}}$. 

\begin{theorem} \label{DP-metaplctic}
Let $\pi_{1}$ be an irreducible admissible genuine representation of $\widetilde{\GL}_{2}(E)$ and let $\pi_2$ be an infinite dimensional irreducible admissible representation of ${\rm GL}_2(F)$. Assume that the central characters $\omega_{\pi_{1}}$ of $\pi_{1}$ and $\omega_{\pi_{2}}$ of $\pi_{2}$ agree on $E^{\times 2} \cap F^{\times}$. 
Fix a non-trivial additive character $\psi$ of $E$ such that $\psi|_{F} = 1$. 
Let $Q = (\pi_{1})_{N(E)}$ be the Jacquet module of $\pi_{1}$. 
Assume that the ``Working Hypothesis \ref{working hypothesis}" holds. 
Then 
\begin{enumerate}
 \item[(A)] Let $\pi_{2} = \Ind_{B(F)}^{\GL_{2}(F)}(\chi)$ be a principal series representation of $\GL_{2}(F)$. 
 Assume that $\Hom_{A(F)} \left( Q, \chi \cdot \delta^{1/2} \right) = 0$. 
 Then
 \[
  \dim \Hom_{{\rm GL}_2(F)} \left( \pi_{1}, \pi_{2} \right) = \dim \Hom_{Z(F)}( (\pi_{1})_{N(E), \psi}, \omega_{\pi_{2}}).
 \]
 \item[(B)] Let $\pi_{1}= \Ind_{\widetilde{B(E)}}^{\widetilde{\GL}_{2}(E)} (\tilde{\tau})$ be a principal series representation of $\widetilde{\GL}_{2}(E)$ and $\pi_2$ a discrete series representation of ${\rm GL}_2(F)$. 
 Let $\pi_2'$ be the finite dimensional representation of $D_{F}^{\times}$ associated to $\pi_{2}$ by the Jacquet-Langlands correspondence. 
 Assume that $\Hom_{\GL_{2}(F)} \left( \Ind_{B(F)}^{\GL_{2}(F)}(\tilde{\tau}), \pi_{2} \right) = 0$. 
 Then
 \[
  \dim \Hom_{{\rm GL}_2(F)} \left( \pi_{1}, \pi_{2} \right) + \dim \Hom_{D_{F}^{\times}} \left( \pi_{1}, \pi_{2}' \right) =  [E^{\times} : F^{\times}E^{\times 2}].
 \]
 \item[(C)] Let $\pi_1$ be an irreducible admissible genuine representation of $\widetilde{\GL}_{2}(E)$ and $\pi_2$ a supercuspidal representation of ${\rm GL}_2(F)$. 
 Let $\pi_{1}'$ be a genuine representation of $\widetilde{\GL}_{2}(E)$ which has the same central character as that of $\pi_{1}$ and as a $\tilde{Z}$-module $(\pi_{1})_{N(E), \psi} \oplus (\pi_{1}')_{N(E), \psi} = \Omega(\omega_{\pi_{1}})$. 
 Let $\pi_2'$ be the finite dimensional representation of $D_{F}^{\times}$ associated to $\pi_2$ by the Jacquet-Langlands correspondence. 
 Then
\[
 \dim \Hom_{{\rm GL}_2(F)} \left( \pi_{1} \oplus \pi_{1}', \pi_{2} \right) + \dim \Hom_{D_{F}^{\times}} \left( \pi_{1} \oplus \pi_{1}', \pi_{2}' \right) =  [E^{\times} : F^{\times}E^{\times 2}].
\] 
\end{enumerate}
\end{theorem}

The strategy to prove this theorem is similar to that in \cite{Prasad92}. We recall it briefly.
Part (A) of this theorem is proved by looking at the Kirillov model of an irreducible admissible genuine representation of $\widetilde{\GL}_{2}(E)$ and its Jacquet module with respect to $N(F)$. 
Part (B) makes use of Mackey theory. 
For the third part (C), we use a trick of Prasad \cite{Prasad92}, where we `transfer' the results of a principal series representation (from the part (B)) which do not belong to principal series. 
Prasad transfers the results from principal series representations to discreet series representations. 
This is done by using character theory and an analog of a result of Casselman and Prasad \cite[~Theorem 2.7]{Prasad92} for $\widetilde{\GL}_{2}(E)$ proved by this author in Section \ref{Cass-Prasad}.
\begin{acknowledgement}
 The author would like to express his indebtness to Professor Dipendra Prasad. 
 Without his continuous encouragement and help this paper would not have appeared in the present form. The author is also thankful to Professor Sandeep Varma for several helpful discussions.
\end{acknowledgement}

\section{Part A of Theorem \ref{DP-metaplctic}}

Let $\pi_{2} = \Ind_{B(F)}^{{\rm GL}_2(F)}(\chi)$ be a principal series representation of ${\rm GL}_2(F)$ where $\chi$ is a character of $A(F)$. By Frobenius reciprocity \cite[~Theorem 2.28]{BerZel76}, we get 
\[
\begin{array}{lcl}
 \Hom_{{\rm GL}_2(F)}(\pi_{1}, \pi_{2}) & = & \Hom_{{\rm GL}_2(F)}(\pi_{1}, \Ind_{B(F)}^{{\rm GL}_2(F)}(\chi)) \\
                                                  & = & \Hom_{A(F)}((\pi_{1})_{N(F)}, \chi . \delta^{1/2}) 
\end{array}
\]
where $(\pi_{1})_{N(F)}$ is the Jacquet module of $\pi_{1}$ with respect to $N(F)$. We can describe $(\pi_{1})_{N(F)}$ by realising $\pi_1$ in the Kirillov model.
Now depending on whether $\pi_1$ is a supercuspidal representation or not, we consider them separately. 

\subsection{Kirillov model and Jacquet module}
Now we describe the Kirillov model of an irreducible admissible genuine representation $\pi$ of $\widetilde{\GL}_2(E)$. 
Let $l : \pi \rightarrow \pi_{N(E), \psi}$ be the canonical map. 
Let $\mathcal{C}^{\infty}(E^{\times}, \pi_{N(E), \psi})$ denote the space of smooth functions on $E^{\times}$ with values in $\pi_{N(E), \psi}$. Define the Kirillov mapping
\[
\mathtt{K} : \pi \longrightarrow \mathcal{C}^{\infty}(E^{\times}, \pi_{N(E), \psi})
\]
given by $v \mapsto \xi_{v}$ where $\xi_{v}(x) = l \left( \pi \left( \left( \begin{matrix} x & 0 \\ 0 & 1 \end{matrix} \right), 1 \right) v \right)$. 
More conceptually, $\pi_{N(E), \psi}$ is a representation space of $\tilde{Z} \cdot N(E)$, and by Frobenius reciprocity, there exists a natural map
\[
 \pi|_{\tilde{B}(E)} \rightarrow \Ind_{ \tilde{Z} \cdot N(E)}^{\tilde{B}(E)} \pi_{N(E), \psi}.
\]
Since $\tilde{B}(E)/\tilde{Z} \cdot N(E)$ can be identified to $E^{\times}$ sitting as $\left\{ \left( \begin{matrix} e & 0 \\ 0 & 1 \end{matrix} \right) : e \in E^{\times} \right\}$ in $\tilde{B}(E)$, we get a map of $\tilde{B}(E)$-modules:
\[
\pi|_{\tilde{B}(E)} \rightarrow C^{\infty}(E^{\times}, \pi_{N(E), \psi}).
\]

We summarize some of the properties of the Kirillov mapping in the following proposition. 
\begin{proposition} \label{prop:Kirillov model}
\begin{enumerate}
\item If $v' = \pi \left( \left( \begin{matrix} a & b \\ 0 & d \end{matrix} \right), 1 \right) v$ for $v \in \pi$ then
\[
\xi_{v'}(x) = (x, d) \psi(bd^{-1}x) \pi \left( \left( \begin{matrix} d & 0 \\ 0 & d \end{matrix} \right), 1 \right) \xi_{v}(ad^{-1}x).
\]
\item For $v \in \pi$ the function $\xi_{v}$ is a locally constant function on $E^{\times}$ which vanishes outside a compact subset of $E$.
\item The map $\mathtt{K}$ is an injective linear map. 
\item The image $\mathtt{K}(\pi)$ of the map $\mathtt{K}$ contains the space $\mathcal{S}(E^{\times}, \pi_{N(E), \psi})$ of smooth functions on $E^{\times}$ with compact support with values in $\pi_{N(E), \psi}$.
\item The Jacquet module $\pi_{N(E)}$ of $\pi$ is isomorphic to $\mathtt{K}(\pi)/\mathcal{S}(E^{\times}, \pi_{N(E), \psi})$.
\item The representation $\pi$ is supercuspidal if and only if $\mathtt{K}(\pi) = \mathcal{S}(E^{\times}, \pi_{N(E), \psi})$.
\end{enumerate}
\end{proposition}
\begin{proof}
Part 1 follows from the definition. The proofs of part 2 and 3 are verbatim those of Lemma 2 and Lemma 3 in \cite{God70}. The proofs of part 4, 5 and 6 follow from the proofs of the corresponding statements of \cite[~Theorem 3.1]{PrRa2000}.
\end{proof}
Since the map $\mathtt{K}$ is injective, we can transfer the action of $\widetilde{\GL}_{2}(E)$ on the space of $\pi$ to $\mathtt{K}(\pi)$ using the map $\mathtt{K}$. The realization of the representation $\pi$ on the space $\mathtt{K}(\pi)$ is called the Kirillov model, on which the action of $\widetilde{B(E)}$ is explicitly given by part 1 in Proposition \ref{prop:Kirillov model}. It is clear that $\mathcal{S}(E^{\times}, \pi_{N(E), \psi})$ is $\widetilde{B(E)}$ stable, which gives rise to the following short exact sequence of $\widetilde{B(E)}$-modules
\begin{equation} \label{Kirillov ses}
0 \rightarrow \mathcal{S}(E^{\times}, \pi_{N(E), \psi}) \rightarrow \mathtt{K}(\pi) \rightarrow \pi_{N(E)} \rightarrow 0.
\end{equation}

\subsection{The Jacquet module with respect to $N(F)$}
In this section, we try to understand the restriction of an irreducible admissible genuine representation $\pi$ of $\widetilde{\GL}_{2}(E)$ to $B(F)$. 
For doing this, we describe the Jacquet module $\pi_{N(F)}$ of $\pi$. 
We utilize the short exact sequence in Equation (\ref{Kirillov ses}) of $\widetilde{B(E)}$-modules arising from the Kirillov model of $\pi$, which is also a short exact sequence of $B(F)$-modules.
By the exactness of the Jacquet functor with respect to $N(F)$, we get the following short exact sequence from Equation (\ref{Kirillov ses}), 
\[
0 \rightarrow \mathcal{S}(E^{\times}, \pi_{N(E), \psi})_{N(F)} \rightarrow \mathtt{K}(\pi)_{N(F)} \rightarrow \pi_{N(E)} \rightarrow 0.
\]
Let us first describe $\mathcal{S}(E^{\times}, \pi_{N(E), \psi})_{N(F)}$, the Jacquet module of $\mathcal{S}(E^{\times}, \pi_{N(E), \psi})$ with respect to $N(F)$. 
\begin{proposition} \label{restriction jacquet}
There exists an isomorphism  
$\mathcal{S}(E^{\times}, \pi_{N, \psi})_{N(F)} \cong \mathcal{S}(F^{\times}, \pi_{N(E), \psi})$
of $F^{\times}$-modules where $F^{\times}$ acts by its natural action on $\mathcal{S}(F^{\times}, \pi_{N(E), \psi})$.
\end{proposition}
The Proposition \ref{restriction jacquet} follows from the proposition below. The author thanks Professor Prasad for suggesting the proof below. 
\begin{proposition} 
Let $\mathcal{S}(E^{\times})$ be a representation space for $N \cong E$ with the action of $N$ given by $(n \cdot f)(x) = \psi(nx) f(x)$ for all $x \in E^{\times}$ where $\psi$ is a non-trivial additive character of $E$ such that $\psi|_{F} =1$. 
Then the restriction map
\begin{equation} \label{restriction map}
\mathcal{S}(E^{\times}) \longrightarrow \mathcal{S}(F^{\times})
\end{equation}
gives the Jacquet module, i.e. the above map realizes $\mathcal{S}(E^{\times})_{N(F)}$ as $\mathcal{S}(F^{\times})$.
\end{proposition}
\begin{proof}
Note that $\mathcal{S}(E^{\times}) \hookrightarrow \mathcal{S}(E)$. 
For a fixed Haar measure $dw$ on $E$, we define the Fourier transform $\mathcal{F}_{\psi} : \mathcal{S}(E) \rightarrow \mathcal{S}(E)$ with respect to the character $\psi$ by
\[
\mathcal{F}_{\psi}(f)(z) := \int_{E} f(w) \psi(zw) \, dw.
\]
As is well known, $\mathcal{F}_{\psi} : \mathcal{S}(E) \rightarrow \mathcal{S}(E)$ is an isomorphism of vector spaces, and the image of $\mathcal{S}(E^{\times})$ can be identified with those functions in $\mathcal{S}(E)$ whose integral on $E$ is zero. 
The Fourier transform takes the action of $N(E)$ on $\mathcal{S}(E^{\times})$ to the restriction of the action of $N(E)$ on $\mathcal{S}(E)$ given by $(n \cdot f) (x) = f(x+n)$. 
Here we have identified $N(E)$ with $E$. 
Note that the maximal quotient of $\mathcal{S}(E)$ on which $N(F)$ acts trivially ($N(F)$ acting by translation on $\mathcal{S}(E)$) can be identified with $\mathcal{S}(F)$ by integration along the fibers (defined below) of the mapping $\phi : E \rightarrow F$ given by $\phi(e) = \frac{e - \bar{e}}{2 \sqrt{d}}$ if $E = F(\sqrt{d})$. 
Note that $\phi(z_1) = \phi(z_2)$ for $z_1, z_2 \in E$ if and only if $z_1 - z_2 \in F$.
We define the integration along the fibers of the map $\phi : E \rightarrow F$, to be denoted by $I$, from  $\mathcal{S}(E) \rightarrow \mathcal{S}(F)$ as follows:
\[
I(f)(y) := \int_{F} f(x + \sqrt{d} y) \, dx \text{ for all } y \in F.
\] 
Clearly $I(f)$ belongs to $\mathcal{S}(F)$. 
Note that $\psi_{\sqrt{d}} = \psi_{\sqrt{d}}|_{F} : x \mapsto \psi(\sqrt{d}x)$ is a non-trivial character of $F$. 
The proposition will follow if we prove the commutativity of the following diagram:
\begin{displaymath}
\xymatrix{ 
\mathcal{S}(E) \ar[r]^{\mathcal{F}_{\psi}} \ar[d]_{{\rm Res}} & \mathcal{S}(E) \ar[d]^{I} \\
\mathcal{S}(F) \ar[r]^{\mathcal{F}_{\psi_{\sqrt{d}}}} & \mathcal{S}(F)   }
\end{displaymath} 
where $\mathcal{F}_{\psi}$ (respectively, $\mathcal{F}_{\psi_{\sqrt{d}}}$) is the Fourier transform on $\mathcal{S}(E)$ (respectively, $\mathcal{S}(F)$) with respect to the character $\psi$ (respectively, $\psi_{\sqrt{d}} =(\psi_{\sqrt{d}})|_{F}$), ${\rm Res}$ denotes the restriction of functions from $E$ to $F$ and $I$ denotes the integration along the fibers mentioned above. 
Recall that $\mathcal{F}_{\psi_{\sqrt{d}}} : \mathcal{S}(F) \rightarrow \mathcal{S}(F)$ is defined by 
\[
\mathcal{F}_{\psi_{\sqrt{d}}} (\phi) (x) := \int_{F} \phi(y) \psi_{\sqrt{d}}(xy) dy = \int_{F} \phi(y) \psi(\sqrt{d} xy) dy \text{ for all } x \in F.
\]

We prove that the above diagram is commutative. 
Let $f \in \mathcal{S}(E)$. 
We want to show that $I \circ \mathcal{F}_{\psi} (f) (y) = \mathcal{F}_{\psi_{\sqrt{d}}} \circ {\rm Res}(f) (y)$ for all $y \in F$. 
We write an element of $E$ as $x + \sqrt{d}y$ with $x, y \in F$. 
We choose a measure $dx$ on $F$ which is self dual with respect to $\psi_{\sqrt{d}}$ in the sense that $\mathcal{F}_{\psi_{\sqrt{d}}} (\mathcal{F}_{\psi_{\sqrt{d}}}(\phi)) (x) = \phi(-x)$ for all $\phi \in \mathcal{S}(F)$ and $x \in F$. 
We identify $E$ with $F \times F$ as a vector space. 
Consider the product measure $dx \, dy$ on $E = F \times F$. 
Using Fubini's theorem we have $\int_{F} \int_{F} \phi(z_{2}) \psi_{\sqrt{d}}(xz_{2}) dz_{2} \, dx = \mathcal{F}_{\psi_{\sqrt{d}}} ( \mathcal{F}_{\psi_{\sqrt{d}}} (\phi))(0) = \phi(0)$  for $\phi \in \mathcal{S}(F)$.
Therefore
\begin{displaymath}
\begin{array}{lcl}
I \circ \mathcal{F}_{\psi} (f) (y) & = & \int_{F} \mathcal{F}_{\psi}(f)(x+\sqrt{d}y)  dx\\
                                               & = & \int_{F} \int_{E=F \times F} f(z_{1} + \sqrt{d} z_{2}) \psi((x + \sqrt{d}y)(z_{1}+ \sqrt{d}z_{2})) dz_{1} \, dz_{2} \, dx \\
                                               & =& \int_{F}  \int_{F} \int_{F} f(z_{1} + \sqrt{d} z_{2}) \psi_{\sqrt{d}} (yz_{1} + x z_{2}) dz_{1} \, dz_{2} \, dx \\
                                               & =& \int_{F} \left( \int_{F} \int_{F} f(z_{1} + \sqrt{d} z_{2}) \psi_{\sqrt{d}} (xz_{2}) dz_{2} \, dx \right) \psi_{\sqrt{d}}(yz_{1}) dz_{1} \\
                                               &=& \int_{F} f(z_{1}) \psi_{\sqrt{d}}(yz_{1}) dz_{1}  \\
                                               &=&  \mathcal{F}_{\psi_{\sqrt{d}}} \circ {\rm Res}(f) (y).
\end{array} 
\end{displaymath}  
This proves the commutativity of the above diagram. 
\end{proof}

\subsection{Completion of the proof of Part A}
First we consider the case when $\pi_{1}$ is a supercuspidal representation of $\widetilde{\GL}_{2}(E)$. 
Then one knows that the functions in the Kirillov model for $\pi_1$ have compact support in $E^{\times}$ and one has $\pi_{1} \cong  \mathcal{S}(E^{\times}, (\pi_{1})_{N(E), \psi})$ as $\widetilde{B(E)}$ modules by Proposition \ref{prop:Kirillov model}. 
Now using Proposition \ref{restriction jacquet} we get the following:
\[
\begin{array}{lcl}
\Hom_{{\rm GL}_2(F)}(\pi_{1}, \pi_{2}) & = & \Hom_{A(F)} \left( (\pi_{1})_{N(F)}, \chi \cdot \delta^{1/2} \right) \\
                        & = & \Hom_{A(F)} \left( \mathcal{S}(E^{\times}, (\pi_{1})_{N(E), \psi})_{N(F)}, \chi \cdot \delta^{1/2} \right) \\ 
                        & = & \Hom_{A(F)} \left( \mathcal{S}(F^{\times}, (\pi_{1})_{N(E), \psi}), \chi \cdot \delta^{1/2} \right) \\ 
\end{array}
\]
Since $\mathcal{S}(F^{\times}, (\pi_{1})_{N, \psi}) \cong \ind_{Z(F)}^{A(F)}(\pi_{1})_{N(E), \psi}$ as $A(F)$-modules, by Frobenius reciprocity \cite[~Proposition 2.29]{BerZel76}, we get the following:
\[
\begin{array}{lcl}
\Hom_{{\rm GL}_2(F)}(\pi_{1}, \pi_{2}) & = & \Hom_{A(F)} \left( \ind_{Z(F)}^{A(F)}(\pi_{1})_{N(E), \psi}, \chi \cdot \delta^{1/2} \right) \\
                        & = & \Hom_{Z(F)} \left( (\pi_{1})_{N(E), \psi}, (\chi \cdot \delta^{1/2})|_{Z(F)} \right) \\
                        & = & \Hom_{Z(F)} \left( (\pi_{1})_{N(E), \psi}, \omega_{\pi_{2}} \right).
\end{array}
\]
This proves part A of Theorem \ref{DP-metaplctic} for $\pi_1$ a supercuspidal representation.

Now we consider the case when $\pi_1$ is not a supercuspidal representation of $\widetilde{\GL}_{2}(E)$. 
Then from Equation (\ref{Kirillov ses}) we get the following short exact sequence of $A(F)$-modules
\[
 0 \rightarrow \mathcal{S}(F^{\times}, (\pi_{1})_{N(E), \psi}) \rightarrow (\pi_{1})_{N(F)} \rightarrow Q \longrightarrow 0.
\]
Now applying the functor $\Hom_{A(F)}(-, \chi . \delta^{1/2})$, we get the following long exact sequence 
\[
\begin{array}{llll}
0 & \rightarrow \Hom_{A(F)} \left( Q, \chi . \delta^{1/2} \right) & \rightarrow \Hom_{A(F)} \left( (\pi_{1})_{N(F)}, \chi . \delta^{1/2} \right) \\
  & \rightarrow \Hom_{A(F)} \left( \mathcal{S}(F^{\times}, (\pi_{1})_{N(E), \psi}), \chi . \delta^{1/2} \right) & \rightarrow \Ext^{1}_{A(F)} \left( Q, \chi . \delta^{1/2} \right) \\
  & \rightarrow \cdots
\end{array}
\]
\begin{lemma}
  $\Hom_{A(F)} \left( Q, \chi . \delta^{1/2} \right) = 0$ if and only if $\Ext^{1}_{A(F)} \left( Q, \chi . \delta^{1/2} \right) = 0$.
\end{lemma}
\begin{proof}
The space $Q$ is finite dimensional and completely reducible. 
So it is enough to prove the lemma for one dimensional representation, i.e., for characters of $A(F)$. 
Moreover one can regard these representations as representation of $F^{\times}$ (after tensoring by a suitable character of $A(F)$ so that it descends to a representation of $A(F)/Z(F) \cong F^{\times}$). 
Then our lemma follows from the following lemma due to Prasad.
\end{proof}
\begin{lemma}
If $\chi_1$ and $\chi_2$ are two characters of $F^{\times}$, then
\[
\dim \Hom_{F^{\times}}(\chi_1, \chi_2) = \dim \Ext_{F^{\times}}^{1}(\chi_1, \chi_2). 
\] 
\end{lemma}
\begin{proof}
Let $\mathcal{O}$ be the ring of integers of $F$ and $\varpi$ a uniformizer of $F$. Since $F^{\times} \cong \mathcal{O}^{\times} \times \varpi^{\Z}$ and $\mathcal{O}^{\times}$ is compact, $\Ext_{F^{\times}}^{i}( \chi_{1}, \chi_{2}) = H^{i} \left( \Z, \Hom_{\mathcal{O}^{\times}}( \chi_{1}, \chi_{2}) \right)$. 
If $\Hom_{\mathcal{O}^{\times}} (\chi_{1}, \chi_{2}) = 0$, then the lemma is obvious. 
Hence suppose that $\Hom_{\mathcal{O}^{\times}}( \chi_{1}, \chi_{2}) \neq 0$. 
Then  $\Hom_{\mathcal{O}^{\times}}( \chi_{1}, \chi_{2})$ is certain one dimensional vector space with an action of $\varpi^{\Z}$. 
If the action of $\varpi^{\Z}$ on $\Hom_{\mathcal{O}^{\times}} (\chi_{1}, \chi_{2})$ is non-trivial then $H^{i} (\Z, \Hom_{\mathcal{O}^{\times}} (\chi_{1}, \chi_{2})) = 0$ for all $i \geq 0$. 
Whereas if the action of $\varpi^{\Z}$ on $\Hom_{\mathcal{O}^{\times}} (\chi_{1}, \chi_{2})$ is trivial, then $H^{0}(\Z, \C) \cong H^{1}(\Z, \C) \cong \C$. 
\end{proof}
We have made an assumption that $\Hom_{A(F)}(Q, \chi . \delta^{1/2}) =0$ and hence by the lemma above $\Ext^{1}_{A(F)}(Q, \chi . \delta^{1/2}) = 0$. 
So in this case 
\[
\begin{array}{lcl}
 \Hom_{A(F)}((\pi_{1})_{N(F)}, \chi . \delta^{1/2}) & \cong & \Hom_{A(F)}(\mathcal{S}(F^{\times}, (\pi_{1})_{N(E), \psi}), \chi . \delta^{1/2}) \\
    & = & \Hom_{Z(F)} \left( (\pi_{1})_{N(E), \psi}, \omega_{\pi_{2}} \right).
\end{array}
\]
Hence
\[
\dim \Hom_{\GL_{2}(F)} \left( \pi_{1}, \pi_{2} \right) = \dim \Hom_{Z(F)} \left( (\pi_{1})_{N(E), \psi}, \omega_{\pi_{2}} \right).
 \]
\begin{remark}
As $Q$ is a finite dimensional representation of $\widetilde{A(E)}$, only finitely many characters of $A(F)$ appear in $Q$. 
For a given $\pi_{1}$ there are only finitely many characters $\chi$ such that $\Hom_{A(F)}(Q, \chi.\delta^{1/2}) \neq 0$. 
We are leaving out at most $2 [E^{\times} : E^{\times 2}]$ many principal series representations $\pi_{2}$ for a given $\pi_{1}$. 
Note that $2 [E^{\times} : E^{\times 2}]$ is the maximum possible dimension of $Q$, i.e. the case of a principal series representation $\pi_{1}$.
\end{remark}

\section{Part B of Theorem \ref{DP-metaplctic}} \label{DP-metaplectic-Part2}
In this section, we consider the case when $\pi_1$ is a principal series representation of $\widetilde{\GL}_{2}(E)$ and $\pi_2$ a discrete series representation of ${\rm GL}_{2}(F)$.

Let $\pi_{1} = \Ind_{\widetilde{B(E)}}^{\widetilde{\GL}_{2}(E)}(\tilde{\tau})$, where $(\tilde{\tau}, V)$ is a genuine irreducible representation of $\tilde{A}=\widetilde{A(E)}$. 
The group $\tilde{A}$ sits in the following central extension
\[
 1 \rightarrow A^{2} \times \{ \pm 1 \} \rightarrow \tilde{A} \xrightarrow{p} A/A^{2} \rightarrow 1,
\]
with $A/A^{2} = E^{\times}/E^{\times 2} \times E^{\times}/E^{\times 2}$ and commutator of two elements $\tilde{a_1}$ and $\tilde{a_2}$ of $\tilde{A}$ whose image in $A/A^{2} = E^{\times}/E^{\times 2} \times E^{\times}/E^{\times 2}$ is $a_1 = (e_1, f_1)$ and $a_2 = (e_2, f_2)$, is
\[
 [ \tilde{a_1}, \tilde{a_2}] = (e_1, f_2)(e_2, f_1) \in \{ \pm 1 \} \subset A^{2} \times \{ \pm 1\},
\]
the product of Hilbert symbols $(e_i, f_j)$ of $E$.
Since the Hilbert symbol is a non-degenerate bilinear form on $E^{\times}/E^{\times 2}$, it follows that
\[
 [ \tilde{a_1}, \tilde{a_2}] : A/A^{2} \times A/A^{2} \rightarrow \{ \pm 1 \}
\]
is also a non-degenerate (skew-symmetric) bilinear form. 
Thus $\tilde{A}$ is closely related to the `usual Heisenberg' groups, and its representation theory is closely related to the representation theory of the `usual Heisenberg' groups.
In particular, given a character $\chi : A^{2} \times \{ \pm 1 \} \rightarrow \C^{\times}$
which is non-trivial on $\{ \pm 1 \}$, there exists a unique irreducible representation of $\tilde{A}$ which contains $\chi$. 
Further, for any subgroup $A_0 \subset A/A^{2}$ for which the commutator map 
$[ \tilde{a_1}, \tilde{a_2} ]$, $a_i \in A_{0}$ is identically trivial, and $A_0$ is maximal for this property, $\tilde{A_0}=p^{-1}(A_0)$ is a maximal abelian subgroup of $\tilde{A}$, and the restriction of an irreducible genuine representation $\tilde{\tau}$ of $\tilde{A}$ to $\tilde{A_0}$ contains all characters of $\tilde{A_0}$ with multiplicity one whose restriction to the center $A^{2} \times \{ \pm 1 \}$ is the central character of $\tilde{\tau}$.
Further, $\tilde{\tau} = \Ind_{\tilde{A}_0}^{\tilde{A}} \chi$ where $\chi$ is any character of $\tilde{A}_0$ appearing in $\tilde{\tau}$.
All the assertions here are consequences of the fact that the inner conjugation action of $\tilde{A}$ on $\tilde{A_0}$ is transitive on the set of characters of $\tilde{A_0}$ with a given restriction on $A^{2} \times \{ \pm 1 \}$; this itself is a consequence of the non-degeneracy of the Hilbert symbol.

It follows that the set of equivalence classes of irreducible genuine representations $\tilde{\tau}$ of $\tilde{A}$ is parametrized by the set of characters of $A^2$, i.e. a pair of characters of $E^{\times 2}$. 

\begin{lemma} \label{all characters}
 The subgroup $\tilde{Z} \cdot A^{2}$ of $\tilde{A}$ is a maximal abelian subgroup. 
 Let $\tilde{\tau}$ be an irreducible genuine representation of $\tilde{A}$. 
 Then $\tilde{\tau}|_{\tilde{Z}}$ contains all the genuine characters of $\tilde{Z}$ which agree with the central character of $\tau$ when restricted to $\tilde{Z}^{2}$.
\end{lemma}
\begin{proof}
By explicit description of commutation relation recalled above it is easy to see that $\tilde{Z} \cdot A^{2}$ is a maximal abelian subgroup of $\tilde{A}$. 
The rest of the statements follow from preceding discussion.
\end{proof}

\begin{proposition} {\rm \cite[~Theorem 2.4]{GPS80}} \label{whittaker models of principal series}
Let $\pi_1 = \Ind_{\widetilde{B(E)}}^{\widetilde{\GL}_{2}(E)}(\tilde{\tau})$ for some irreducible genuine representation $\tilde{\tau}$ of $\tilde{A}$. 
Then 
\[
 (\pi_{1})_{N, \psi} \cong \Omega(\pi_{1}) \cong \tilde{\tau}|_{\tilde{Z}}.
\]
\end{proposition}

Now as in \cite{Prasad92}, we use Mackey theory to understand its restriction to ${\rm GL}_2(F)$. We have $\widetilde{\GL}_{2}(E)/\widetilde{B(E)} \cong \mathbb{P}_{E}^{1}$ and this has two orbits under the left action of ${\rm GL}_2(F)$. 
One of the orbits is closed, and naturally identified with $\mathbb{P}_{F}^{1} \cong {\rm GL}_2(F)/B(F)$. 
The other orbit is open, and can be identified with $\mathbb{P}_{E}^{1} - \mathbb{P}_{F}^{1} \cong {\rm GL}_2(F)/E^{\times}$. 
By \index{Mackey theory} Mackey theory, we get the following exact sequence of ${\rm GL}_2(F)$-modules:
\begin{equation} \label{p.s. ses}
0 \rightarrow \ind_{E^{\times}}^{{\rm GL}_2(F)} (\tilde{\tau}'|_{E^{\times}}) \rightarrow \pi_{1} \rightarrow \Ind_{B(F)}^{{\rm GL}_2(F)}(\tilde{\tau}|_{B(F)}\delta^{1/2}) \rightarrow 0,
\end{equation}
where $\tilde{\tau}'|_{E^{\times}}$ is the representation of $E^{\times}$ obtained from the embedding $E^{\times} \hookrightarrow \tilde{A}$ which comes from conjugating the embedding $E^{\times} \hookrightarrow \GL_{2}(F) \hookrightarrow \widetilde{\GL}_{2}(E)$. 
We now identify $E^{\times}$ with its image inside $\tilde{A}$ which is given by $x \mapsto \left( \left( \begin{matrix} x & 0 \\ 0 & \bar{x} \end{matrix} \right), \epsilon(x) \right)$ where $\bar{x}$ is the non-trivial $Gal(E/F)$-conjugate of $x$ and $\epsilon(x) \in \{ \pm 1 \}$. 
Now let $\pi_{2}$ be any irreducible admissible representation of ${\rm GL}_2(F)$. 
By applying the functor $\Hom_{{\rm GL}_2(F)}(-, \pi_{2})$ to the short exact sequence (\ref{p.s. ses}), we get the following long exact sequence:
\begin{equation} \label{p.s. les}
\begin{array}{lllll}
0 & \rightarrow & \Hom_{{\rm GL}_2(F)} [\Ind_{B(F)}^{{\rm GL}_2(F)}(\tilde{\tau}|_{B(F)}\delta^{1/2}), \pi_{2}] & \rightarrow & \Hom_{{\rm GL}_2(F)} [\pi_{1}, \pi_{2}]  \\
  & \rightarrow & \Hom_{{\rm GL}_2(F)} [\ind_{E^{\times}}^{{\rm GL}_2(F)}(\tilde{\tau}'|_{E^{\times}}), \pi_{2}] & \rightarrow & \Ext_{{\rm GL}_2(F)}^{1}[\Ind_{B(F)}^{{\rm GL}_2(F)}(\tilde{\tau}|_{B(F)} \delta^{1/2}), \pi_{2}] \\
  & \rightarrow  & \cdots & &
\end{array}
\end{equation}
From \cite[~Corollary 5.9]{Prasad90} we know that 
\begin{center}
$\Hom_{{\rm GL}_2(F)} [\Ind_{B(F)}^{{\rm GL}_2(F)}(\chi.\delta^{1/2}), \pi_{2}] = 0$ \\
$\Updownarrow$ \\
$\Ext_{{\rm GL}_2(F)}^{1}[\Ind_{B(F)}^{{\rm GL}_2(F)}(\chi.\delta^{1/2}), \pi_{2}] = 0.$
\end{center}
Since $\tilde{\tau}|_{B(F)}$ factors through $T(F)$, which is direct sum of $[E^{\times} : E^{\times 2}]$ characters of $T(F)$, we can use above result of Prasad with $\chi$ replaced by $\tilde{\tau}|_{B(F)}$.  Then from the exactness of (\ref{p.s. les}), it follows that 
\begin{center}
$\Hom_{{\rm GL}_2(F)} [\pi_{1},\pi_{2}] = 0$ \\
$\Updownarrow$\\
$\Hom_{{\rm GL}_2(F)} [\Ind_{B(F)}^{{\rm GL}_2(F)}(\tilde{\tau}|_{B(F)} \delta^{1/2}), \pi_{2}] = 0 \text{ and }
\Hom_{{\rm GL}_2(F)} [\ind_{E^{\times}}^{{\rm GL}_2(F)} (\tilde{\tau}'|_{E^{\times}}),\pi_{2}] = 0.$
\end{center}
Note that the representation $\Ind_{B(F)}^{{\rm GL}_2(F)}(\tilde{\tau}|_{B(F)})$ consists of finitely many principal series representations of ${\rm GL}_2(F)$. 
We have made the assumption that $ \Hom_{{\rm GL}_2(F)} [\Ind_{B(F)}^{{\rm GL}_2(F)}(\tilde{\tau}|_{B(F)}), \pi_{2}] = 0$, it follows that  
\[
 \Ext_{{\rm GL}_2(F)}^{1}[\Ind_{B(F)}^{{\rm GL}_2(F)}(\tilde{\tau}.\delta^{1/2}), \pi_{2}] = 0.
\]
This gives
\[ 
\begin{array}{lll}
 \Hom_{{\rm GL}_2(F)}[\pi_{1}, \pi_{2}] & \cong & \Hom_{{\rm GL}_2(F)} [\ind_{E^{\times}}^{{\rm GL}_2(F)}(\tilde{\tau}'|_{E^{\times}}), \pi_{2}] \\
                         & \cong & \Hom_{E^{\times}} [\tilde{\tau}'|_{E^{\times}}, \pi_{2}|_{E^{\times}}]         
\end{array}
\]
The following lemma describes $\tilde{\tau}'|_{E^{\times}}$.
\begin{lemma}
If we identify $E^{\times}$ with its image $\left\{ \left( \left( \begin{matrix} x & 0 \\ 0 & \bar{x} \end{matrix} \right), \epsilon(x) \right) \mid x \in E^{\times} \right\}$ inside $\tilde{A}$ as above then the subgroup $E^{\times} \cdot \tilde{A^2}$ inside $\tilde{A}$ is a maximal abelian subgroup. 
Moreover, $\tilde{\tau}'|_{E^{\times}}$ contains all the characters of $E^{\times}$ which are same as $\omega_{\tilde{\tau}}|_{E^{\times 2}}$ when restricted to $E^{\times 2}$, where $\omega_{\tilde{\tau}}$ is the central character of $\tilde{\tau}$.
\end{lemma}
\begin{proof}
From the explicit cocycle description and the non-degeneracy of quadratic Hilbert symbol, it is easy to verify that $E^{\times} \cdot \tilde{A^2}$ is a maximal abelian subgroup of $\tilde{A}$. The rest follows from the discussion preceding Lemma \ref{all characters}.
\end{proof}

As $\pi_{2}$ is a discrete series representation, it is not always true (unlike what happens in case of a principal series representation) that any character of $E^{\times}$, whose restriction to $F^{\times}$ is the same as the central character of $\pi_{2}$, appears in $\pi_{2}$. 
Let $\pi_{2}'$ be the finite dimensional representation of $D_{F}^{\times}$ associated to $\pi_{2}$ by the Jacquet-Langlands correspondence. 
Considering the left action of $D_{F}^{\times}$ on $\mathbb{P}_{E}^{1} \cong \widetilde{\GL}_{2}(E)/\widetilde{B(E)}$ induced by $D_{F}^{\times} \hookrightarrow \widetilde{\GL}_{2}(E)$ it is easy to verify that $\mathbb{P}_{E}^{1} \cong D_{F}^{\times} / E^{\times}$. 
Then by Mackey theory, the principal series representation $\pi_{1}$ when restricted to $D_{F}^{\times}$, becomes isomorphic to $\ind_{E^{\times}}^{D_{F}^{\times}}(\tilde{\tau}'|_{E^{\times}})$. 
\[
\begin{array}{lll}
\Hom_{D_{F^{\times}}}[\pi_{1}, \pi_{2}']  & \cong & \Hom_{D_{F^{\times}}}[\ind_{E^{\times}}^{D_{F}^{\times}}(\tilde{\tau}'|_{E^{\times}}), \pi_{2}'] \\
                                 & \cong & \Hom_{E^{\times}}(\tilde{\tau}'|_{E^{\times}}, \pi_{2}'|_{E^{\times}})
\end{array}
\]

In order to prove 
\begin{equation} 
 \dim \Hom_{{\rm GL}_2(F)}[\pi_{1}, \pi_{2}] + \dim \Hom_{D_{F}^{\times}}[\pi_{1}, \pi_{2}'] = [E^{\times} : F^{\times}E^{\times 2}]
\end{equation}
we shall prove 
\begin{equation} \label{V pi2 pi2'}
 \dim \Hom_{E^{\times}} [\tilde{\tau}'|_{E^{\times}}, \pi_{2}|_{E^{\times}}] + \dim \Hom_{E^{\times}}(\tilde{\tau}'|_{E^{\times}}, \pi_{2}'|_{E^{\times}}) = [E^{\times} : F^{\times}E^{\times 2}].
\end{equation}
By Remark 2.9 in \cite{Prasad92}, a character of $E^{\times}$ whose restriction to $F^{\times}$ is the same as the central character of $\pi_{2}$ appears either in $\pi_{2}$ with multiplicity one or in $\pi_{2}'$ with multiplicity one, and exactly one of the two possibilities hold. 
Note that we are assuming that the two embeddings of $E^{\times}$, one via $\GL_{2}(F)$ and other via $D_{F}^{\times}$ are conjugate in $\widetilde{\GL}_{2}(E)$. 
Then the left hand side of Equation (\ref{V pi2 pi2'}) is the same as the number of characters of $E^{\times}$ appearing in $(\tilde{\tau}, V)$ which upon restriction to $F^{\times}$ coincide with the central character of $\pi_{2}$, which equals $\dim \Hom_{F^{\times}}(\tilde{\tau}|_{F^{\times}}, \omega_{\pi_2})$.
We are reduced to the following lemma.
\begin{lemma}
 Let $(\tilde{\tau}, V)$ be an irreducible genuine representation of $\tilde{A}$ and, let $\chi$ be a character of $Z(F) = F^{\times}$ such that $\chi|_{E^{\times 2} \cap F^{\times}} = \tilde{\tau}|_{E^{\times 2} \cap F^{\times}}$. Then 
 \[
  \dim \Hom_{F^{\times}}(\tilde{\tau}, \chi) = [E^{\times} : F^{\times}E^{\times 2}].
 \]
\end{lemma}
\begin{proof}
Note that $E^{\times 2} \cap F^{\times} = Z^{\times 2} \cap F^{\times}$. 
From Proposition \ref{whittaker models of principal series}, $\tilde{\tau}|_{\tilde{Z}} \cong \Omega(\omega_{\pi_{1}})$. 
If a character $\mu \in \Omega(\omega_{\pi_{1}})$ is specified on $F^{\times}$ then it is specified on $F^{\times}E^{\times 2}$. Therefore the number of characters in $\Omega(\omega_{\pi_{1}})$ which agree with $\chi$ when restricted to $F^{\times}$ is equal to $[E^{\times} : F^{\times}E^{\times 2}]$. \qedhere
 \end{proof}

\section{A theorem of Casselman and Prasad} \label{Cass-Prasad}
As mentioned in the introduction, we use results of part (B) involving principal series representation and `transfer' it to the other cases, as stated in part (C) which involves restriction of the two representations.
To make such a transfer possible Prasad used a result which says that if two irreducible representations of $\GL_{2}(E)$ have the same central characters then the difference of their character is a smooth function on $\GL_{2}(E)$.
We will need a similar theorem for $\widetilde{\GL}_{2}(E)$, which we prove in this section.
In order to do this, we recall a variant of a theorem of Rodier which is true for covering groups in general; this variant is proved by the author \cite{Shiv15}. Let us first recall some facts about germ expansions, restricted only to $\widetilde{\SL}_{2}(E)$.

For any non-zero nilpotent orbit in $\mathfrak{sl}_{2}(E)$ there is a lower triangular nilpotent matrix $Y_{a} = \left( \begin{matrix} 0 & 0 \\ a & 0 \end{matrix} \right)$ such that $Y_{a}$ belongs to the nilpotent orbit. 
For a given non-zero nilpotent orbit, the element $a$ is uniquely determined modulo $E^{\times 2}$.
We write $\mathcal{N}_{a}$ for the nilpotent orbit which contains $Y_{a}$.
Thus the set of all non-zero nilpotent orbits is $\{ \mathcal{N}_{a} \mid a \in E^{\times}/E^{\times 2} \}$.

Let $\tau$ be an irreducible admissible genuine representation of $\widetilde{\SL}_{2}(E)$. 
Recall that for an irreducible admissible genuine representation $\tau$ of $\widetilde{\SL}_{2}(E)$, the character distribution $\Theta_{\tau}$ is a smooth function on the set of regular semisimple elements. 
The Harish-Chandra-Howe character expansion of $\Theta_{\tau}$ in a neighbourhood of identity is given  as follows:
\[
\Theta_{\tau} \circ \exp = c_{0}(\tau) + \sum_{a \in E^{\times}/E^{\times 2}} c_{a}(\tau) \cdot \widehat{\mu}_{\mathcal{N}_{a}}
\]
where $c_{0}(\tau), c_{a}(\tau)$ are constants and $\widehat{\mu}_{\mathcal{N}_{a}}$ is the Fourier transform of a suitably chosen $\SL_{2}(E)$-invariant (under the adjoint action) measure on $\mathcal{N}_{a}$. 

Fix a non-trivial additive character $\psi$ of $E$. 
Define a character $\chi$ of $N$ by $\chi \left( \begin{matrix} 1 & x \\ 0 & 1 \end{matrix} \right) = \psi(x)$. 
For $a \in E^{\times}$ we write $\psi_{a}$ for the character of $E$ given by $\psi_{a}(x) = \psi(ax)$.
We write $(N, \psi)$ for the non-degenerate Whittaker datum $(N, \chi)$. 
It can be seen that the set of conjugacy classes of non-degenerate Whittaker data has a set of representatives $\{ (N, \psi_{a}) \mid a \in E^{\times}/E^{\times 2} \}$.

By the proof of the main theorem in \cite{Shiv15}, the bijection between $\{ \mathcal{N}_{a} \mid a \in E^{\times}/E^{\times 2} \}$ and $\{ (N, \psi_{a}) \mid a \in E^{\times}/E^{\times 2} \}$ given by $\mathcal{N}_{a} \leftrightarrow  (N, \psi_{a})$ satisfies the following property:
$c_{a} \neq 0$ if and only if the representation $\tau$ of $\widetilde{\SL}_{2}(E)$ admits a non-zero $(N, \psi_{a})$-Whittaker functional. \\

It follows from \cite[~Theorem 4.1]{GHPS79} that for any non-trivial additive character $\psi'$ of $N$, the dimension of the space of $(N,\psi')$-Whittaker functionals for $\tau$ is at most one. 
Therefore, from the theorem of Rodier, as extended in \cite{Shiv15}, each $c_{a}(\tau)$ is either 1 or 0 depending on whether $\tau$ admits a non-zero Whittaker functional corresponding to the non-degenerate Whittaker datum $(N, \psi_{a})$ or not. 
\begin{remark} \label{semisimple G tilde}
 Let $\tilde{G}$ be a topological central extension of a connected reductive group $G$ by $\mu_{r}$, a cyclic group of order $r$. For $g \in \tilde{G}$ there exists a semisimple element $g_{s} \in \tilde{G}$ such that $g$ belongs to any conjugation invariant neighbourhood of $g_s \in \tilde{G}$.

\end{remark}

Let $\tau_{1}$ and $\tau_{2}$ be two irreducible admissible genuine representations of $\widetilde{\SL}_{2}(E)$. 
As $\widetilde{\{ \pm 1 \}}$ is the center of $\widetilde{\SL}_{2}(E)$ which are the only non-regular semisimple elements of $\widetilde{\SL}_{2}(E)$ and $\Theta_{\tau_1}, \Theta_{\tau_2}$ are given by smooth functions at regular semisimple points,
if $\Theta_{\tau_{1}} - \Theta_{\tau_{2}}$ is a smooth function in a neighbourhood of the identity then it is smooth function on the whole of $\widetilde{\SL}_{2}(E)$ provided $\tau_{1}, \tau_{2}$ have the same central characters.\\

For any non-trivial additive character $\psi'$ of $E$, let us assume that $\tau_{1}$ admits a non-zero Whittaker functional for $(N,\psi')$ if and only if $\tau_{2}$ does so too. 
Under this assumption $c_{a}(\tau_{1}) = c_{a}(\tau_{2})$ for all $a \in E^{\times}/E^{\times 2}$. Then we have the following result.
\begin{theorem} \label{Cass-Prasad-SL2}
Let $\tau_{1}, \tau_{2}$ be two irreducible admissible genuine representations of $\widetilde{\SL}_{2}(E)$ with the same central characters. 
For a non-trivial additive character $\psi'$ of $E$ assume that $\tau_{1}$ admits a non-zero Whittaker functional with respect to $(N,\psi')$ if and only if $\tau_{2}$ admits a non-zero Whittaker functional with respect to $(N,\psi')$. 
Then $\Theta_{\tau_1} - \Theta_{\tau_2}$ is constant in a neighbourhood of identity and hence extends to a smooth function on all of $\widetilde{\SL}_{2}(E)$.
\end{theorem}
Using Theorem \ref{Cass-Prasad-SL2}, we prove an extension of a theorem of Casselman-Prasad \cite[~Theorem 5.2]{Prasad92}. 
\begin{theorem} \label{Cass-Prasad-GL2}
Let $\psi$ be a non-trivial character of $E$. 
Let $\pi_{1}$ and $\pi_{2}$ be two irreducible admissible genuine representations of $\widetilde{\GL}_{2}(E)$ with the same central characters such that $(\pi_{1})_{N, \psi} \cong (\pi_{2})_{N, \psi}$ as $\tilde{Z}$-modules. 
Then $\Theta_{\pi_{1}} - \Theta_{\pi_{2}}$ initially defined on regular semi-simple elements of $\widetilde{\GL}_{2}(E)$ extends to a smooth function on all of $\widetilde{\GL}_{2}(E)$.
\end{theorem}
\begin{proof}
We already know that $\Theta_{\pi_{1}}$ and $\Theta_{\pi_{2}}$ are smooth on the set of regular semisimple elements, so is $\Theta_{\pi_{1}} - \Theta_{\pi_{2}}$. 
To prove the smoothness of $\Theta_{\pi_{1}} - \Theta_{\pi_{2}}$ on whole of $\widetilde{\GL}_{2}(E)$, we need to prove the smoothness at every point in $\tilde{Z}$. 
As $\tilde{Z}$ is not the center of $\widetilde{\GL}_{2}(E)$, the smoothness at the identity is not enough to imply the smoothness at every point in $\tilde{Z}$. 
Note that $\tilde{Z}$ is the center of $\widetilde{\GL}_{2}(E)_{+} := \tilde{Z} \cdot \widetilde{\SL}_{2}(E)$ and  $\widetilde{\GL}_{2}(E)_{+}$ is an open and normal subgroup of $\widetilde{\GL}_{2}(E)$ of index $[E^{\times} : E^{\times 2}]$. 

Choose irreducible admissible genuine representations $\tau_{1}$ and $\tau_{2}$ of $\widetilde{\SL}_{2}(E)$ and characters $\mu_1, \mu_2$ of $\tilde{Z}$ such that
\begin{equation}
\pi_{1} = \ind_{\widetilde{\GL}_{2}(E)_{+}}^{\widetilde{\GL}_{2}(E)}(\mu_1 \tau_{1}) \,\, \text{ and } \,\, \pi_{2} = \ind_{\widetilde{\GL}_{2}(E)_{+}}^{\widetilde{\GL}_{2}(E)}(\mu_2 \tau_{2})
\end{equation}
where $\mu_i \tau_{i}$ for $i =1,2$ denotes the irreducible representation of $\widetilde{\GL}_{2}(E)_{+}$ on which $\tilde{Z}$ acts by $\mu_i$ and $\widetilde{\SL}_{2}(E)$ acts by $\tau_{i}$. 
It follows that
\begin{equation} \label{noname-1}
\pi_{1}|_{\widetilde{\GL}_{2}(E)_{+}} = \bigoplus_{a \in E^{\times}/E^{\times 2}} (\mu_1 \tau_{1})^{a} \,\, \text{ and } \,\, \pi_{2}|_{\widetilde{\GL}_{2}(E)_{+}} = \bigoplus_{a \in E^{\times}/E^{\times 2}} (\mu_2 \tau_{2})^{a},
\end{equation}
where we abuse notation to let $a$ denote the matrix $\left( \begin{matrix} a & 0 \\ 0 & 1 \end{matrix} \right)$. 
Observe that $(\mu_1 \tau_1)^{a} = \mu_{1}^{a} \tau_{1}^{a}$ with $\mu_{1}^{a}(\tilde{z}) = (a, z) \mu_{1}(\tilde{z})$, where $\tilde{z} \in \tilde{Z}$ lies above $z \in Z$; in particular, all the characters $\mu_{1}^{a}$ for $a \in E^{\times}/E^{\times 2}$, are distinct.
From the identity (\ref{noname-1}) we find that 
\begin{equation}
 (\pi_1)_{N(E), \psi} = \bigoplus_{a \in E^{\times}/E^{\times 2}} \mu_{1}^{a} (\tau_{1}^{a})_{N(E), \psi} \,\, \text{ and } \,\, (\pi_2)_{N(E), \psi} = \bigoplus_{a \in E^{\times}/E^{\times 2}} \mu_{2}^{a} (\tau_{2}^{a})_{N(E), \psi}
\end{equation}
Since $(\pi_1)_{N, \psi} \cong (\pi_2)_{N, \psi}$ as $\tilde{Z}$-modules, in particular, the part corresponding to $\mu^{a}$-eigenspaces are isomorphic for all $a \in E^{\times}/E^{\times 2}$. 
Therefore $\mu_{1} = \mu_{2}^{b}$ for some $b \in E^{\times}/E^{\times 2}$.
Since $\pi_2 = \ind_{\widetilde{\GL}_{2}(E)_{+}}^{\widetilde{\GL}_{2}(E)}(\mu_{2} \tau_{2}) =
\ind_{\widetilde{\GL}_{2}(E)_{+}}^{\widetilde{\GL}_{2}(E)}(\mu_{2}^{b} \tau_{2}^{b})$, by changing $\tau_{2}$ by $\tau_{2}^{b}$, we can assume that $\pi_{1} = \ind_{\widetilde{\GL}_{2}(E)_{+}}^{\widetilde{\GL}_{2}(E)} (\mu \tau_{1})$, and $\pi_{2} = \ind_{\widetilde{\GL}_{2}(E)_{+}}^{\widetilde{\GL}_{2}(E)}( \mu \tau_{2})$. 
Now $(\pi_{1})_{N(E), \psi} \cong (\pi_{2})_{N(E), \psi}$ as $\tilde{Z}$-modules translates into
$(\tau_{1}^{a})_{N(E), \psi} \cong (\tau_{2}^{a})_{N(E), \psi}$ for all $a \in E^{\times}/E^{\times 2}$. Therefore, by Theorem \ref{Cass-Prasad-SL2}, $\Theta_{\tau_{1}^{a}} - \Theta_{\tau_{2}^{a}}$ is constant in a neighbourhood of the identity for all $a \in E^{\times}/E^{\times 2}$.

Let $\Theta_{\rho, g}$ denote the character expansion of an irreducible admissible representation $\rho$ in a neighbourhood of the point $g$, then

\[
\Theta_{\pi_{1}, \tilde{z}} = \sum_{a \in E^{\times}/E^{\times 2}} \Theta_{(\mu \tau_{1})^{a}, \tilde{z}} = \sum_{a \in E^{\times}/E^{\times 2}} \mu^{a}(\tilde{z}) \Theta_{\tau_{1}^{a}, 1}
\]

and 

\[
\Theta_{\pi_{2}, \tilde{z}} = \sum_{a \in E^{\times}/E^{\times 2}} \Theta_{(\mu \tau_{2})^{a}, \tilde{z}} = \sum_{a \in E^{\times}/E^{\times 2}} \mu^{a}(\tilde{z}) \Theta_{\tau_{2}^{a}, 1}.
\]
This proves that $\Theta_{\pi_1} - \Theta_{\pi_2}$ is a constant function on regular semi-simple points in some neighbourhood of $\tilde{z}$ for all $\tilde{z} \in \tilde{Z} \subset \widetilde{\GL}_{2}(E)$, and therefore it extends to a smooth function in that neighbourhood of $\tilde{z}$. Thus $\Theta_{\pi_1} - \Theta_{\pi_2}$, which is initially defined on regular semi-simple elements of $\widetilde{\GL}_{2}(E)$ extends to a smooth function on all of $\widetilde{\GL}_{2}(E)$.
\end{proof}
\begin{corollary}
Let $\pi_{1}, \pi_{2}$ be two irreducible admissible genuine representations of $\widetilde{\GL}_{2}(E)$ with the same central character such that $(\pi_{1})_{N, \psi} \cong (\pi_{2})_{N, \psi}$ as $\tilde{Z}$-modules. 
Let $H$ be a subgroup of $\widetilde{\GL}_{2}(E)$ that is compact modulo center. 
Then there exist finite dimensional representations $\sigma_{1}, \sigma_{2}$ of $H$ such that 
\[
\pi_{1}|_{H} \oplus \sigma_{1} \cong \pi_{2}|_{H} \oplus \sigma_{2}.
\]
\end{corollary}
In other words, this corollary says that the virtual representation $(\pi_{1} - \pi_{2})|_{H}$ is finite dimensional and hence the multiplicity of an irreducible representation of $H$ in $(\pi_{1} - \pi_{2})|_{H}$ will be finite.

\section{Part C of Theorem \ref{DP-metaplctic}} \label{DP-metaplectic-Part3}
Let $\pi_{1}$ be an irreducible admissible genuine representation of $\widetilde{\GL}_{2}(E)$. 
We take another admissible genuine representation $\pi_{1}'$ having the same central character as that of $\pi_{1}$ and satisfying $(\pi_{1})_{N(E), \psi} \oplus (\pi_{1}')_{N(E), \psi} \cong \Omega( \omega_{\pi_{1}} )$ as $\tilde{Z}$-modules. 
From Proposition \ref{whittaker models of principal series}, if $\pi_{1}$ is a principal series representation then we can take $\pi_{1}' = 0$. 
It can be seen that if $\pi_{1}$ is not a principal series representation then $(\pi_{1})_{N(E), \psi}$ is a proper $\tilde{Z}$-submodule of $\Omega( \omega_{\pi_{1}} )$ forcing $\pi_{1}' \neq 0$. 
In particular, if $\pi_{1}$ is one of the Jordan-H{\"o}lder factors of a reducible principal series representation then one can take $\pi_{1}'$ to be the other Jordan-H{\"o}lder factor of the principal series representation. 
It should be noted that for a supercuspidal representation $\pi_{1}$ we do not have any obvious choice for $\pi_{1}'$.\\

Let $\pi_{2}$ be a supercuspidal representation of ${\rm GL}_2(F)$. 
To prove Theorem \ref{DP-metaplctic} in this case, we use character theory and deduce the result by using the result of restriction of a principal series representation of $\widetilde{\GL}_{2}(E)$ which has already been proved in Section \ref{DP-metaplectic-Part2}. 
We can assume, if necessary after twisting by a character of $F^{\times}$, that $\pi_{2}$ is a minimal representation. 
Recall that an irreducible representation $\pi_{2}$ of $\GL_{2}(F)$ is called minimal if the conductor of $\pi_{2}$ is less than or equal to the conductor of $\pi_{2} \otimes \chi$ for any character $\chi$ of $F^{\times}$. 
By a theorem of Kutzko \cite{Kutzko78}, a minimal supercuspidal representation $\pi_2$ of $\GL_{2}(F)$ is of the form $\ind_{\mathcal{K}}^{ {\rm GL}_2(F)}(W_2)$, where $W_2$ is a representation of a maximal compact modulo center subgroup $\mathcal{K}$ of ${\rm GL}_2(F)$. 
By Frobenius reciprocity, 
\[
\begin{array}{lll}
\Hom_{{\rm GL}_2(F)} \left( \pi_{1} \oplus \pi_{1}', \pi_{2} \right) & = & \Hom_{{\rm GL}_2(F)} \left( \pi_{1} \oplus \pi_{1}', \ind_{\mathcal{K}}^{{\rm GL}_2(F)}(W_2) \right) \\
                        & = & \Hom_{\mathcal{K}} \left( (\pi_{1} \oplus \pi_{1}')|_{\mathcal{K}}, W_2 \right). 
\end{array}
\]
To prove Theorem \ref{DP-metaplctic}, it suffices to prove that:
\[
 \dim \Hom_{\mathcal{K}}[(\pi_{1} \oplus \pi_{1}')|_{\mathcal{K}}, W_2] + \dim \Hom_{D_{F}^{\times}}[\pi_{1} \oplus \pi_{1}', \pi_{2}'] = [E^{\times} : F^{\times}E^{\times 2}].
\]
For any (virtual) representation $\pi$ of $\widetilde{\GL}_{2}(E)$, let $m(\pi, W_2) = \dim \Hom_{\mathcal{K}}[\pi|_{\mathcal{K}}, W_2]$ and $m(\pi,\pi_{2}') = \dim \Hom_{D_{F}^{\times}}[\pi, \pi_{2}']$. With these notations we will prove: 
\begin{equation} \label{Equation:Main-1}
 m(\pi_{1} \oplus \pi_{1}',W_2) + m(\pi_{1} \oplus \pi_{1}',\pi_{2}') = [E^{\times} : F^{\times}E^{\times 2}].
\end{equation}

Let $Ps$ be an irreducible principal series representation of $\widetilde{\GL}_{2}(E)$ whose central character $\omega_{Ps}$ is same as the central character $\omega_{\pi_{1}}$ of $\pi_{1}$ (it is clear that there exists one such). By Proposition \ref{whittaker models of principal series}, we know that $(Ps)_{N(E), \psi} \cong \Omega(\omega_{Ps})$ as a $\tilde{Z}$-module. On the other hand, the representation $\pi_{1}'$ has been chosen in such a way that $(\pi_{1})_{N(E), \psi} \oplus (\pi_{1}')_{N(E), \psi} = \Omega(\omega_{\pi_{1}})$ as $\tilde{Z}$-module. Then, as a $\tilde{Z}$-module we have
\[
(\pi_{1} \oplus \pi_{1}')_{N(E), \psi} = (\pi_{1})_{N(E), \psi} \oplus (\pi_{1}')_{N(E), \psi} = \Omega(\omega_{\pi_{1}}) = \Omega({\omega_{Ps}}) = (Ps)_{N(E), \psi}.
\]
We have already proved in Section \ref{DP-metaplectic-Part2} that
\[
 m(Ps,W_2) + m(Ps,\pi_{2}') = [E^{\times} : F^{\times}E^{\times 2}].
\]
In order to prove Equation (\ref{Equation:Main-1}), we prove 
\begin{equation} \label{Equation:Main-2}
m(\pi_{1} \oplus \pi_{1}' -Ps,W_2) + m(\pi_{1} \oplus \pi_{1}' -Ps,\pi_{2}') = 0.
\end{equation}
The relation in Equation (\ref{Equation:Main-2}) follows from the following theorem:
\begin{theorem}
Let $\Pi_{1}, \Pi_{2}$ be two genuine representations of $\widetilde{\GL}_{2}(E)$ of finite length with a central character which is same for $\Pi_{1}$ and $\Pi_{2}$, and such that $(\Pi_{1})_{N(E), \psi} \cong (\Pi_{2})_{N(E), \psi}$ as $\tilde{Z}$-modules for a non-trivial additive character $\psi$ of $E$. Let $\pi_{2}$ be an irreducible supercuspidal representation of $\GL_{2}(F)$ such that the central characters $\omega_{\Pi_{1}}$ of $\Pi_{1}$ and $\omega_{\pi_{2}}$ of $\pi_{2}$ agree on $F^{\times} \cap E^{\times 2}$. Let $\pi_{2}'$ be the finite dimensional representation of $D_{F}^{\times}$ associated to $\pi_{2}$ by the Jacquet-Langlands correspondence. Then
\[
m(\Pi_{1} - \Pi_{2}, \pi_{2}) + m(\Pi_{1} - \Pi_{2}, \pi_{2}') = 0.
\]
\end{theorem}
We will use character theory to prove this relation following \cite{Prasad92} very closely. First of all, by Theorem \ref{Cass-Prasad-GL2}, $\Theta_{\Pi_{1} - \Pi_{2}}$ is given by smooth function on $\widetilde{\GL}_{2}(E)$. Now we recall the Weyl integration formula for ${\rm GL}_{2}(F)$.\\
\subsection{Weyl integration formula} \label{Weyl_I_F}
\begin{lemma} {\rm \cite[Formula 7.2.2]{JL70}} \label{Weyl-Integration-Formula} \\
For a smooth and compactly supported function $f$ on ${\rm GL}_{2}(F)$ we have
\begin{equation} \label{Expression-W-I-F}
\int_{{\rm GL}_{2}(F)} f(y) dy = \sum_{E_{i}} \int_{E_{i}} \bigtriangleup(x) \left( \frac{1}{2} \int_{E_{i} \backslash {\rm GL}_{2}(F)} f(\bar{g}^{-1}x \bar{g}) \, d\bar{g} \right) \, dx
\end{equation}
where the $E_{i}$'s are representatives for the distinct conjugacy classes of maximal tori in ${\rm GL}_{2}(F)$ and 
\[
\bigtriangleup(x) = \left| \left| \dfrac{(x_{1}- x_{2})^{2}}{x_{1}x_{2}} \right| \right|_{F}
\]
where $x_{1}$ and $x_{2}$ are the eigenvalues of $x$. 
\end{lemma}
We will use this formula to integrate the function $f(x) = \Theta_{\Pi_{1} - \Pi_{2}} \cdot \Theta_{W_{2}}(x)$ on $\mathcal{K}$ which is extended to ${\rm GL}_{2}(F)$ by setting it to be zero outside $\mathcal{K}$. In addition, we also need the following result of Harish-Chandra, cf. \cite[~Proposition 4.3.2]{Prasad92}.
\begin{lemma} [Harish-Chandra]  \label{lemma:HC}
Let $F(g) = (g v, v)$ be a matrix coefficient of a supercuspidal representation $\pi$ of a reductive $p$-adic group $G$ with center $Z$. Then the orbital integrals of $F$ at regular non-elliptic elements vanish. Moreover, the orbital integral \index{orbital integral} of $F$at a regular elliptic element $x$ contained in a torus $T$ is given by the formula
\begin{equation} \label{Harish-chandra theorem}
\int_{T \backslash G} F(\bar{g}^{-1}x\bar{g}) d\bar{g} = \dfrac{(v,v) \cdot \Theta_{\pi} (x)}{d(\pi) \cdot {\rm vol}(T/Z)},
\end{equation}
where $d(\pi)$ denotes the formal degree of the representation $\pi$ (which depends on a choice of Haar measure on $T \backslash G$).
\end{lemma}
Since $\pi_{2}$ is obtained by induction from $W_{2}$, a matrix coefficient \index{matrix coefficient} of $W_{2}$ (extended to ${\rm GL}_{2}(F)$ by setting it to be zero outside $\mathcal{K}$) is also a matrix coefficient of $\pi_{2}$. It follows that 
\begin{enumerate}
\item for the choice of Haar measure on ${\rm GL}_{2}(F)/F^{\times}$ giving $\mathcal{K}/F^{\times}$ measure 1, we have 
\[
\dim W_{2} = d(\pi_{2}),
\]
\item for a separable quadratic field extension $E_{i}$ of $F$ and a regular elliptic element $x$ of $\GL_{2}(E)$ which generates $E_{i}$, and for the above Haar measure $d \bar{g}$,
\begin{equation} \label{noname-2}
\int_{E_{i}^{\times} \backslash {\rm GL}_{2}(F)} \Theta_{W_{2}}(\bar{g}^{-1} x \bar{g}) d\bar{g} = \dfrac{\Theta_{\pi_{2}}(x)}{{\rm vol}(E_{i}^{\times}/F^{\times})}. 
\end{equation}
\end{enumerate}

\subsection{Completion of the proof of Theorem \ref{DP-metaplctic}}
We recall the following important observation from Section \ref{Weyl_I_F} and Theorem \ref{Cass-Prasad-GL2}:
\begin{enumerate}
\item the virtual representation $(\Pi_{1} - \Pi_{2})|_{\mathcal{K}}$ is finite dimensional,
\item $\Theta_{W_{2}}$ is also a matrix coefficient of $\pi_{2}$ (extended to $\GL_{2}(F)$ by zero outside $\mathcal{K}$), 
\item there is Haar measure on $\GL_{2}(F)/F^{\times}$ giving $\vol(\mathcal{K}/F^{\times})=1$ such that the Equation \ref{noname-2} is satisfied.
\item the orbital integral in Equation (\ref{Harish-chandra theorem}) vanishes if $T$ is maximal split torus.  
\end{enumerate}
Let $E_{i}$'s be the quadratic extensions of $F$. Then these observations together with Lemma \ref{lemma:HC}, imply the following
\[
\begin{array}{rcl}
m(\Pi_{1} - \Pi_{2}, W_2) &=& \dfrac{1}{\vol(\mathcal{K}/F^{\times})} \bigint\limits_{\mathcal{K}/F^{\times}} \Theta_{\Pi_{1} - \Pi_{2}} \cdot \Theta_{W_{2}}(x) \, dx \\
                                                   &=& \dfrac{1}{\vol(\mathcal{K}/F^{\times})} \bigint\limits_{\GL_{2}(F)/F^{\times}} \Theta_{\Pi_{1} - \Pi_{2}} \cdot \Theta_{W_{2}}(x) \, dx \\
                                                   &=& \dfrac{1}{\vol(\mathcal{K}/F^{\times})} \sum\limits_{E_{i}} \bigint\limits_{E_{i}^{\times}/F^{\times}} \bigtriangleup(x) \left[ \dfrac{1}{2} \bigint\limits_{E_{i}^{\times} \backslash \GL_{2}(F)} \Theta_{\Pi_{1} - \Pi_{2}} \cdot \Theta_{W_{2}}(\bar{g}^{-1}x \bar{g}) \, dg \right] dx \\
                                                   &=& \sum\limits_{E_{i}} \dfrac{1}{2 \vol(E_{i}^{\times}/F^{\times})} \bigint\limits_{E_{i}^{\times}/F^{\times}} \left( \bigtriangleup \cdot \Theta_{\Pi_{1} - \Pi_{2}} \cdot \Theta_{\pi_{2}} \right) (x) dx.
\end{array}                                                    
\]
Similarly, we have the equality
\[
 m(\Pi_{1} - \Pi_{2}, \pi_{2}') = \sum_{E_i}\dfrac{1}{2 \vol(E_{i}^{\times}/F^{\times})} \bigint\limits_{E_{i}^{\times}/F^{\times}} \left( \bigtriangleup . \Theta_{\Pi_{1} - \Pi_{2}} . \Theta_{\pi_{2}'} \right) (x)dx.
\]
Note that $E_{i}$'s correspond to quadratic extensions of $F$ and the embeddings of $\GL_{2}(F)$ and $D_{F}^{\times}$ have been fixed so that the working hypothesis (as stated in the introduction of this chapter) is satisfied, i.e. the embeddings of the $E_{i}$'s in $\GL_{2}(F)$ and in $D_{F}^{\times}$ are conjugate in $\widetilde{\GL}_{2}(E)$. Then the value of $\Theta_{\Pi_{1} - \Pi_{2}}(x)$ for $x \in E_{i}$, does not depend on the inclusion of $E_{i}$ inside $\widetilde{\GL}_{2}(E)$, i.e. on whether inclusion is via $\GL_{2}(F)$ or via $D_{F}^{\times}$. Now using the relation $\Theta_{\pi_{2}}(x) = -\Theta_{\pi_{2}'}(x)$ on regular elliptic elements $x$ \cite[~Proposition 15.5]{JL70}, we conclude the following, which proves the Equation (\ref{Equation:Main-2})

\[
 m(\Pi_{1} - \Pi_{2}, W_2) + m(\Pi_{1} - \Pi_{2}, \pi_{2}') = 0. 
\]
\section{A remark on higher multiplicity}
We have shown that the restriction of an irreducible admissible representation of $\widetilde{\GL}_{2}(E)$, for example a principal series representation, to the subgroup $\GL_{2}(F)$ has multiplicity more than one. 
Given the important role multiplicity one theorems play, it would be desirable to modify the situation so that multiplicity one might be true. 
One natural way to do this is to decrease the larger group, and increase the smaller group. 
In this section we discuss some natural subgroups of the group $\widetilde{\GL}_{2}(E)$ which can be used, but unfortunately, it still does not help one to achieve multiplicity one situation. We discuss this modification in this section in some detail.

Let us take the subgroup of $\widetilde{\GL}_{2}(E)$ which is generated by $\GL_{2}(F)$ and $\tilde{Z}$.
We will prove that this subgroup also fails to achieve multiplicity one for the restriction problem from $\widetilde{\GL}_{2}(E)$ to $\GL_{2}(F) \cdot \tilde{Z}$. 
Let $H = \GL_{2}(F) \subset H_{+} = Z \cdot \GL_{2}(F) \subset \GL_{2}(E)$.
We will show that the restriction of an irreducible admissible representation of $\widetilde{\GL}_{2}(E)$ to the subgroup  $\tilde{H_+}$ has higher multiplicity.
Note that the subgroups $\tilde{Z}$ and $\GL_{2}(F)$ do not commute but $\tilde{Z^2}$ commutes with $\GL_{2}(F)$. 
In fact, the commutator relation is given by
\begin{equation} \label{commutator relation}
 [\tilde{e}, \tilde{g}] = (e, \det g)_{E} \in \{ \pm 1 \} \subset \widetilde{\GL}_{2}(E),
\end{equation}
where $\tilde{e} \in \tilde{Z}$ and $\tilde{g} \in \widetilde{\GL}_{2}(F)$ lying over elements $e \in Z$ and $g \in \GL_{2}(F)$ respectively, and $(-,-)_{E}$ denotes the Hilbert symbol for the field $E$.
The lemma below proves that the center of $\tilde{H_+}$ is $\widetilde{Z^{2}F^{\times}}$.
\begin{lemma} \label{Hilbert symbol lemma}
 For an element $e \in E^{\times}$, the map $F^{\times} \rightarrow \{ \pm 1 \}$ defined by $f \mapsto (e, f)_{E}$ is trivial if and only if $e \in F^{\times} E^{\times 2}$.
\end{lemma}
\begin{proof}
Let $( \cdot, \cdot)_{E}$ and $(\cdot, \cdot)_{F}$ denote the Hilbert symbol of the field $E$ and $F$ respectively. For $e \in E^{\times}$ and $f \in F^{\times}$, the following is well known \cite{Bender73}
\[
 (e, f)_{E} = (N_{E/F}(e), f)_{F},
\]
where $N_{E/F}$ is the norm map of the extension $E/F$.
Therefore, if $(e,f)_{E} = 1$ is true for all $f \in F^{\times}$, then by the non-degeneracy of the Hilbert symbol $(\cdot, \cdot)_{F}$ one will have $N_{E/F}(e) \in F^{\times 2}$.
The inverse image of $F^{\times 2}$ under the norm map $N_{E/F}$ is now seen to be $E^{\times 2}F^{\times}$ since this subgroup surjects onto $F^{\times 2}$ under the norm mapping,
and contains the kernel $\{ z/\bar{z} = z^{2}/z\bar{z} : z \in E^{\times} \}$ of $N_{E/F}$.
\end{proof}
Let $\sigma$ be an irreducible admissible representation of $\GL_{2}(F)$. 
By the commutator relation (\ref{commutator relation}), we have
\[
 a (g, \epsilon)a^{-1} = (g, \chi_{a}(g) \epsilon),
\]
where $\chi_{a}(g)$ is a character on $\GL_{2}(F)$ given by $\chi_{a}(g) = (a, \det g)_{E}$.
Therefore, the conjugation action by $a \in Z$ takes $\sigma$ to the quadratic twist $\sigma \otimes \chi_{a}$ where $\chi_{a}$ is given by $x \mapsto (x,a)_{E}$. 
We have the following lemma which easily follows from Clifford theory.
\begin{lemma}
Let $\tilde{H_0} = \tilde{Z^{2}} \cdot \GL_{2}(F)$.
Let $\sigma$ be an irreducible admissible representation of $\GL_{2}(F)$. 
Assume that $\sigma \otimes \chi_{a} \ncong \sigma$ for any non-trivial element $a \in E^{\times}/F^{\times}E^{\times 2}$. 
Fix a genuine character $\eta$ of $\tilde{Z^{2}}$ such that $\eta|_{F^{\times} \cap \tilde{Z^2}} = \omega_{\sigma}|_{F^{\times} \cap \tilde{Z^2}}$. 
Then $\rho = \Ind_{\tilde{H_0}}^{\tilde{H_+}} (\eta \sigma)$ is an irreducible representation of $\tilde{H_+}$. The representation $\rho$ is the only irreducible representation of $\tilde{H_+}$ whose central character restricted to $\tilde{Z^2}$ is $\eta$ and also contains $\sigma$. 
Moreover, $\rho|_{\tilde{H_0}} \cong \bigoplus_{a \in E^{\times}/F^{\times}E^{\times 2}} \eta(\sigma \otimes \chi_{a})$. 
In particular, by Lemma \ref{Hilbert symbol lemma}, the restriction of $\rho$ to $\tilde{H_0}$ is multiplicity free.
\end{lemma}

Note that if $\sigma$ is a principal series representation of $\GL_{2}(F)$ which is not of the form $Ps(\chi_1, \chi_2)$ with $\chi_1 / \chi_2$ a quadratic character, then such principal series representation of $\GL_{2}(F)$ have no non-trivial self twist. 
Let $\pi$ be an irreducible admissible genuine representation of $\widetilde{\GL}_{2}(E)$ such that $\dim \Hom_{\GL_{2}(F)}(\pi, \sigma) \geq 2$.
Let $\eta$ be the central character of $\pi$.
Note that the central character of any irreducible representation of $\tilde{H}_{+}$, which is contained in $\pi$, agrees with $\eta$ when restricted to $\tilde{Z^2}$.
Let $\rho = \Ind_{\tilde{H_0}}^{\tilde{H_+}} (\eta \sigma)$ as in the previous lemma.
The representation $\rho$ is the only representation of $\tilde{H}_{+}$ which appears in $\pi$ and contains $\sigma$.
So the multiplicity of such a principal series representation $\sigma$ of $\GL_{2}(F)$ in the restriction of an irreducible admissible genuine representation of $\widetilde{\GL}_{2}(E)$ is same as the multiplicity of the corresponding irreducible representation of $\tilde{H_+}$, i.e.
$\dim \Hom_{\tilde{H}_{+}}(\pi, \rho) = \dim \Hom_{\GL_{2}(F)}(\pi, \sigma) \geq 2$. Thus we conclude that the restriction of representations of $\widetilde{\GL}_{2}(E)$ to $\tilde{H}_{+}$ has higher multiplicity.

On the other hand, let us take the group $G = \{ g \in \GL_{2}(E) : \det g \in F^{\times} E^{\times 2} \}$. Note that this subgroup $G$ contains $\GL_{2}(E)_{+} = Z \cdot \SL_{2}(E)$.
We will prove that the pair $(\tilde{G}, \GL_{2}(F))$ also fails to achieve multiplicity one for restriction problem from $\tilde{G}$ to $\GL_{2}(F)$.
From the commutation reltion \eqref{commutator relation}, it follows that the center of the group $\tilde{G}$ is $\widetilde{F^{\times} Z^{2}}$.
Recall that the restriction from $\widetilde{\GL}_{2}(E)$ to $\widetilde{\GL}_{2}(E)_{+}$ is multiplicity free and $\tilde{G} \supset \widetilde{\GL}_{2}(E)_{+}$, thus the restriction from $\widetilde{\GL}_{2}(E)$ to $\tilde{G}$ is also multiplicity free.
Let $\pi$ be an irreducible admissible genuine representation of $\widetilde{\GL}_{2}(E)$ and $\rho$ be an irreducible admissible genuine representation of $\tilde{G}$ such that $\rho \hookrightarrow \pi|_{\tilde{G}}$. Then we have
\[
 \pi|_{\tilde{G}} = \bigoplus_{a \in E^{\times}/ F^{\times} E^{\times 2}} \rho^{a}.
\]
For $a_1 \neq a_2$ in $E^{\times}/F^{\times}E^{\times 2}$,  $\rho^{a_1} \ncong \rho^{a_2}$. 
In fact, the central characters of $\rho^{a_1}$ and $\rho^{a_2}$ are different when restricted to $F^{\times}$. 

Let $\pi$ be an irreducible admissible genuine representation of $\widetilde{\GL}_{2}(E)$ and $\sigma$ an irrducible admissible representation of $\GL_{2}(F)$ such that 
\[
 \dim \Hom_{\GL_{2}(F)}(\pi, \sigma) \geq 2.
\]
If $\Hom_{\GL_{2}(F)}(\rho^{a_1}, \sigma) \neq 0$ then $\Hom_{\GL_{2}(F)}(\rho^{a_2}, \sigma) = 0$ for $a_2 \neq a_1$ in $E^{\times}/F^{\times}E^{\times 2}$, since the central character of $\rho^{a_2}$ restricted to $F^{\times}$ will be different from the central character of $\sigma$.
Thus there exists only one $a \in E^{\times}/F^{\times}E^{\times 2}$ such that $\Hom_{\GL_{2}(F)}(\rho^{a}, \sigma) \neq 0$. 
We can assume that $\Hom_{\GL_{2}(F)}(\rho, \sigma) \neq 0$.
We have
\[
 \Hom_{\GL_{2}(F)} (\rho, \sigma) = \Hom_{\GL_{2}(F)} (\pi, \sigma)
\]
and hence $\dim \Hom_{\GL_{2}(F)} (\rho, \sigma) \geq 2$.

\end{document}